\documentclass{article}

\usepackage[preprint]{neurips_2026}

\usepackage[utf8]{inputenc} 
\usepackage[T1]{fontenc}    
\usepackage{hyperref}       
\usepackage{url}            
\usepackage{booktabs}       
\usepackage{amsfonts}       
\usepackage{nicefrac}       
\usepackage{microtype}      
\usepackage{xcolor}  

\usepackage{amsmath,amsfonts,amssymb}
\usepackage{mathtools}
\usepackage{graphicx}
\usepackage{algorithm}
\usepackage{algpseudocode}
\usepackage{cleveref}
\usepackage{xcolor}
\usepackage{soul}
\usepackage{cancel}
\newcommand{\floor}[1]{\left\lfloor #1 \right\rfloor}
\newcommand{\ceil}[1]{\left\lceil #1 \right\rceil}
\newtheorem{theorem}{Theorem}[section]

\newtheorem{corollary}[theorem]{Corollary}

\newtheorem{lemma}[theorem]{Lemma}

\newcommand\numberthis{\addtocounter{equation}{1}\tag{\theequation}}
\usepackage{mathrsfs}

\title{Sharp Optimal Algorithm for Derivative-Free Stochastic Convex Optimization in One Dimension}

\usepackage[affil-it]{authblk}

\author[1]{Alexandra Carpentier}
\author[1]{Chloé Rouyer}
\author[2]{Alexandre Tsybakov}
\author[3]{Arya Akhavan}

\affil[1]{Institut für Mathematik, Universität Potsdam, Germany}
\affil[2]{CREST, ENSAE, IP Paris, France} 
\affil[3]{Department of Statistics, University of Oxford, United Kingdom \vspace{0.3cm}}
\affil[ ]{  \hspace{-0.7cm}\footnotesize\ttfamily \textup{\{carpentier,rouyer\}@uni-potsdam.de \quad
alexandre.tsybakov@ensae.fr \quad
arya.akhavan@stats.ox.ac.uk}}

\begin{document}

\maketitle

\begin{abstract}


Stochastic convex optimization is a classical problem with well-understood guarantees under first-order feedback. In contrast, for zero-order optimization with noisy function evaluations, a logarithmic gap has persisted between known upper bounds and the $\Omega(1/\sqrt{T})$ lower bound, even in the one-dimensional case. In this work, we study the problem of minimizing a convex function $f : [0,1] \to [0,1]$ using a zero-order oracle with subGaussian noise. We propose a computationally efficient algorithm that achieves the optimal $O(1/\sqrt{T})$ convergence rate, matching the lower bound.
The result closes the existing gap in one dimension, providing the first sharp rate guarantee in this setting.

\end{abstract}

\section{Introduction}

We study stochastic convex optimization in a setting where only noisy function evaluations are available. At each round $t$, the learner selects a query point $x_t$ in the interval $[0,1]$ and observes $f(x_t) + \xi_t$, where the function evaluation $f(x_t)$ is corrupted by subGaussian noise $\xi_t$. Such zeroth-order feedback arises naturally in applications where gradients are unavailable or expensive to compute. The goal is to identify a point $\hat x$ with near-minimal function value using a limited number $T$ of queries. The performance is evaluated in terms of simple regret, which measures the distance between the function evaluated in $\hat x$ and the true function minimum: $r_T = f(\hat x) - \min_{x \in [0, 1]} f(x)$.
Despite the apparent simplicity of this setting—particularly in one dimension—the optimal statistical rate of convergence for zeroth-order optimization is not fully resolved.

 Information-theoretic limits for this setting are well studied. Minimax lower bound for the rate of decay of the simple regret that scales as $1/\sqrt{T}$ is proved in \cite{polyak1990optimal}. Namely, it follows from \cite{polyak1990optimal} that there exist constants $c_1>0$ and $c_2\in (0,1)$ such that
\begin{align}\label{intro1}
  \inf_{\hat{x}} \sup_{f\in \mathcal{F}} \mathbb{P}_f\Big(f(\hat{x}) - f^* \ge \frac{c_1}{\sqrt{T}} \Big)\ge c_2,  
\end{align}
where $\mathbb{P}_f$ is the probability measure, with which is distributed $(x_t,y_t)_{t=1}^T$, and the infimum is taken over all, possibly randomized, estimators $\hat{x}$. In fact, \cite{polyak1990optimal} established \eqref{intro1} in a stronger form, for $\mathcal{F}$ being the class of strongly convex and smooth functions in any dimension $d$. Later work \cite{shamir2013complexity,akhavan2020exploiting,akhavan2024gradient} derived lower bounds with explicit dependence on the dimension that scale as $d/\sqrt{T}$, in a weaker form than in \eqref{intro1} (for the expected simple regret  rather than for the probability).
 Papers \cite{jamieson2012}, \cite{duchi2015optimal} provided lower bounds with the rate $\sqrt{d/T}$, again for the expected regret.

While the lower bounds are obtained without any logarithmic factors, it is not the case for the available upper bounds. Thus, it remains unknown what is the exact optimal rate for $r_T$ in the setting described above. The stream of work on the upper bounds was mainly focused on improving the dependency on the dimension \cite{agarwal-rakhlin2011,lattimore-gyorgy21a,lattimore-gyorgy23a,fokkema2024online,carpentier2025simple}, starting from the $d^{16}{\rm polylog}(T,d)/\sqrt{T}$ rate in \cite{agarwal-rakhlin2011} and going down to $d^{1.5}{\rm polylog}(T,d)/\sqrt{T}$ in \cite{fokkema2024online}. The ${\rm polylog}(T,d)$ factor was not the main issue of this line of work and was not always explicitly stated. 
Some of these results are obtained for the cumulative regret, and it remains an open question to what extent the lower bound cited above (proved for the simple regret) is accurate in this case. In the one-dimensional setting that we are dealing with here, explicit upper bounds for cumulative regret are provided in \cite{agarwal-rakhlin2011} and \cite{lattimore-gyorgy21a}. Both propose computationally efficient algorithms. In one dimension, the simple regret of the method proposed in \cite{lattimore-gyorgy21a} scales as $(\log T)^{2}/\sqrt{T}$. Under the additional assumption that $f$ is Lipschitz, \cite{agarwal-rakhlin2011} proves that the rate $(\log T)^{3/2}/\sqrt{T}$ of simple regret can be achieved in one-dimensional case.  

A related line of work deals with the adversarial bandit setting \cite{Bubeck15a,bubeck2018,lattimore2020bandit,bubeck2021kernel}. These results are not directly comparable with our setting since the sum of function evaluation and random noise $\xi_t$ cannot be considered as an evaluation of a convex function.
In these works, the cumulative regret is shown to be of the order  ${\rm poly}(d){\rm polylog}(T,d)\sqrt{T}$. In particular, for dimension $d=1$ \cite{Bubeck15a} obtains the rate $\log(T)\sqrt{T}$ using a non-constructive approach via Bayesian minimax duality. 


In the present paper, we propose a novel algorithm for zeroth-order stochastic convex optimization in one dimension. 
We analyze the simple regret and derive both high-probability and in expectation guarantees, which both achieve the optimal rate $1/\sqrt T$,  closing the long-standing gap with the lower bound of \cite{polyak1990optimal}. 
The high-probability result is presented in Theorem \ref{thm:main}. 
We show that the proposed algorithm achieves $O\big(\sqrt{\frac{\log(1/\delta)}{T}} (\log\log(1/\delta))^2\big)$ simple regret guarantee with probability at least $1-\delta$.
The expected simple regret is analyzed in Corollary \ref{cor:expected_sr} and we derive a $O\big(\frac{1}{\sqrt{T}}\big)$ upper bound. 
These results rely on the Splitting Algorithm, which is introduced in \Cref{sec:splitting_algorithm}. It utilizes a geometric grid to ensure that the range of the function shrinks sufficiently at each step, effectively by-passing the logarithmic dependencies typically suffered when working with more uniform grids.


\section{Problem Setting}

We consider the problem of minimizing a convex function \( f: [0,1] \rightarrow [0,1] \). Denote by $\mathcal{F}$ the class of all such convex functions. We set
\[ f^* = \min_{x \in [0,1]} f(x). \]
Assume that for \( t = 1, \ldots, T \):
\begin{itemize}
    \item The learner chooses a point \( x_t \in [0,1] \).
    \item The learner observes \( y_t = f(x_t) + \xi_t \), where $\mathcal{\xi}_t$'s are independent random noise variables. 
\end{itemize}
For each $t\ge 2$, the chosen point $x_t$ is allowed to depend on the past observations $(x_i,y_i)_{i=1}^{t-1}$. We assume in what follows that for each \( t = 1, \ldots, T \) the noise $\xi_t$ is independent of $x_t,(x_i,y_i)_{i=1}^{t-1}$ and $1$-subGaussian, that is,
\[
\forall \, v\in \mathbf{R}, \, t = 1, \ldots, T: \quad \mathbb{E} \exp(v\xi_t) \le \exp(v^2/2).
\]
After \( T \) queries, the learner outputs an estimator  \( \hat{x} \in [0,1] \) measurable with respect to $(x_t,y_t)_{t=1}^T$. The aim of the learner is to construct $\hat x$ such that, on an event of probability \( 1 - \delta \), where $\delta\in (0,1)$, the simple regret
$ r_T = f(\hat{x}) - f^* $
is as small as possible.


\section{Splitting Algorithm}
\label{sec:splitting_algorithm}
To tackle this problem, we propose a subroutine (the Splitting Algorithm, \ref{alg:splitting}) that reduces the size of the searched interval and that will be repeatedly called by a meta-algorithm. When reducing the interval $[I_-, I_+]$ to its sub-interval $[I_-', I_+']$, the Splitting Algorithm also achieves a significantly smaller range of function $f$ on $[I_-', I_+']$ while keeping small the distance to the global function minimum on~$[0, 1]$.

Assume that we are given an interval $[I_-, I_+]$ such that the following holds.

\vspace{2mm}

{\bf{Assumption}$(\epsilon,\Delta,[I_-, I_+])$.}\label{assump:bound}
{\it 
Interval \([I_-, I_+] \subseteq [0,1]\) satisfies the conditions
\[ \min_{x\in [I_-, I_+] } f(x) - f^* \leq \epsilon 
\quad \text{and} \quad 
 \max_{x\in [I_-, I_+] } f(x) - \min_{x\in [I_-, I_+] } f(x)  \leq \Delta, \]
for some \(\epsilon \ge 0, \Delta>0\).}

\vspace{2mm}

In particular, the interval $[0,1]$ satisfies this assumption with $\epsilon=0$, $\Delta=1$. In what follows, we will often set without loss of generality \([I_-,I_+] = [0,1]\).   
The reduction is obtained by rescaling from $f(x)$ to $f(I_- + x(I_+ - I_-))$. 

Define the set of points in $[0,1]$:
\[
\mathcal{G} = \left\{ \frac{1}{2^i}, 1-\frac{1}{2^i}, \forall i \in \{0,\ldots,g\} \right\},
\]
where 
\begin{equation}
    g = \bar g(\Delta) = \lfloor \log_2(T \Delta^2)\rfloor +1. \label{def:g}
\end{equation}
Here, $\lfloor u \rfloor$ denotes the maximal integer less than $u$.
The cardinality of \(\mathcal{G}\) is
$
|\mathcal{G}| = 2g + 1.
$
For \(\delta > 0, c_{\mathrm{sampl}} > 0\), and $\bar g(\Delta)$ defined in \eqref{def:g}, set

\begin{equation}
N = \bar{N}(\Delta,\delta) = \floor{{2}c_{\mathrm{sampl}}^{-2}\frac{\log\Big({2}(2\bar g(\Delta)+1)/\delta\Big) (2\bar g(\Delta)+1)^2}{\Delta^2}} +1.
\label{def:N}
\end{equation}

\begin{algorithm}[]
\caption{Splitting Algorithm}
\label{alg:splitting}
\begin{algorithmic}[1]
\State {\bf input:} An interval \([I_-, I_+]\), a constant \(\Delta > 0\) corresponding to the one in Assumption$(\epsilon,\Delta)$,
a target probability $\delta >0$, constants $c_{\mathrm{sampl}}, c_{\mathrm{cond}} >0.$
\State {\bf initialization:} Rescale $f$ such that $[I_-, I_+] = [0, 1]$ using 
\[
{x \gets \frac{x-I_-}{I_+ - I_-}.}
\]
\ForAll {\(x \in \mathcal{G}\)}
    \State Sample \(x\) a total of \(N\) times
    \State Compute \(\hat{f}_x\) as the average of sampled observations
\EndFor

\State \textbf{define}:
{$$\mathcal{G}_- = \left\{ x \in \mathcal{G} : x<1/2, \, \tau_x > c_{\mathrm{cond}}\frac{\Delta}{2g+1}
\right\},$$
where $\tau_x=\hat{f}_{x} - \hat{f}_{2x}$ if $x\ne 0$ and $\tau_x=\hat{f}_{0} - \hat{f}_{2^{-g}}$ if $x= 0$.
}
\State \textbf{define}:
{$$\mathcal{G}_+ = \left\{ x \in \mathcal{G} : x> 1/2, \, \tau_x' > c_{\mathrm{cond}}\frac{\Delta}{2g+1}
\right\},$$
where $\tau_x'=\hat{f}_{x} - \hat{f}_{2x - 1}$ if $x\ne 1$ and $\tau_x'=\hat{f}_{1} - \hat{f}_{1-2^{-g}}$ if $x= 1$.
}

\State \textbf{define}: 
{
$$I'_- = 
\begin{cases} 
0, & \text{if } \mathcal{G}_- = \emptyset, \\
2^{-g}, & \text{if } \max \mathcal{G}_- = 0, \\
\max \mathcal{G}_-, & \text{otherwise},
\end{cases}
\qquad
I'_+ = 
\begin{cases} 
1, & \text{if } \mathcal{G}_+ = \emptyset, \\
1-2^{-g}, & \text{if } \min \mathcal{G}_+ = 1, \\
\min \mathcal{G}_+, & \text{otherwise}.
\end{cases}
$$
}


\State \Return \([I'_- , I'_+]\)
\end{algorithmic}
\end{algorithm}

Consider the procedure presented in \Cref{alg:splitting}. We call it the Splitting Algorithm. The idea of the algorithm is to estimate $f$ on the grid $\mathcal{G}$, then extract points $x$ from $\mathcal{G}$ with large enough gap between the estimator at $x$ and at its closest neighbor in $\mathcal{G}$ and define the new interval \([I'_- , I'_+]\) as the interval between the minimal and maximal elements of the extracted set. The following lemma describes a shrinkage property of this algorithm.

\begin{lemma}\label{lem:splitting}
    Assume that $\delta\in (0,1/2)$, $c_{\mathrm{sampl}} \leq c_{\mathrm{cond}}/4$, $c_{\mathrm{cond}} \leq 1/3$,  and $g\ge 2$. Let the input interval \([I_-,I_+]\) of \Cref{alg:splitting} satisfy Assumption$(\epsilon,\Delta,[I_-, I_+])$
    with some 
    \(\Delta>0, \epsilon \ge 0\). 
    Then with probability at least \(1-\delta\), \Cref{alg:splitting} outputs an interval \([I_-',I_+']\)  satisfying Assumption$(\epsilon',\Delta',[I_-', I_+'])$
    with 
    \[\Delta' = \left[\frac{3}{4} \lor \left(1 - \frac{c_{\mathrm{cond}}}{2(2g+1)}\right)\right] \Delta, \qquad \text{and} \qquad \epsilon' = \epsilon + 2^{-g+2} \Delta. \]
\end{lemma}


\section{Meta-algorithm}

We now define a meta-algorithm that acts by applying the Splitting Algorithm in epochs numbered $r=1,2,\dots, R$.
For any \(\delta
\in(0,1)\) , $\Delta > 0$, introduce the notation
\[\psi(\Delta) = 2^{-\bar g(\Delta)+ 2}, ~~~\bar \delta(\Delta) =  \frac{\delta}{\Delta^2 T}.\] 
Set
\( [I_-^{(0)}, I_+^{(0)}] = [0,1],~~~\epsilon_0 = 0,~~~\Delta_0 = 1,~~~\delta_0 = \frac{\delta}{T} \)
and define, for any integer $r\ge 1$,
\[g_{r-1} = \bar g (\Delta_{r-1}) = \lfloor \log_2(T \Delta_{r-1}^2) \rfloor + 1, \qquad \delta_{r} =   \bar \delta( \Delta_r) =  \frac{\delta}{\Delta_r^2 T},\]
\begin{align}\label{def:Delta-r}
  \Delta_{r} &= \left[\frac{3}{4} \lor \left(1 - \frac{c_{\mathrm{cond}}}{2(2g_{r-1} + 1)}\right)\right] \Delta_{r-1},
  \end{align}
  with $c_{\mathrm{cond}}>0$ as in Lemma~\ref{lem:splitting}, and
  \begin{align*}
\epsilon_r &= \epsilon_{r-1} + 2^{-g_{r-1} + 2} \Delta_{r-1} =  \epsilon_{r-1} + \psi(\Delta_{r-1}) \Delta_{r-1}. 
\end{align*}

\begin{algorithm}
\caption{Meta-Algorithm}
\label{alg:meta}
\begin{algorithmic}[1]
\State {\bf input:} $\delta \in (0,1/3), \bar{C} >0$.
\State {\bf initialisation:} \(r=0\)
\While{\(\Delta_r \geq \bar \Delta:=\bar{C}\sqrt{\frac{\log(1/\delta)}{T}} {(\log\log(1/\delta))^2}\)}
    \State 
    Apply \textbf{Splitting Algorithm} with \([I_-^{(r)},I_+^{(r)}], \  \Delta_{r}, \ \delta_{r}\) and collect \([I_-^{(r +1)},I_+^{(r+1 )}]\)
\State \(r \gets r + 1\)
\EndWhile
\State {\bf write} $R$ for the last executed epoch, equivalently
$R=\max\{r\ge0:\Delta_r\ge\bar\Delta\}$
\State \Return any \(\hat{x}\in [I_-^{(R+1)},I_+^{(R+1)}]\)
\end{algorithmic}
\end{algorithm}

The meta-algorithm returns \(\hat{x}\) with the properties described in the next theorem. 
\begin{theorem}[High-probability simple regret]\label{thm:main}
Let $\delta\in (0,1/3)$, and { let $c_{\mathrm{sampl}}$, $c_{\mathrm{cond}}$ be as in Lemma~\ref{lem:splitting}.}  
Then there exists \(\bar{C}>0\) independent of $T,\delta$ 
such that Algorithm \ref{alg:meta} is such that its budget does not exceed \(T\) and such that it outputs \(\hat{x}\) satisfying with probability at least \(1-\delta\)
the inequality
\[ f(\hat{x}) - f^* 
\leq
\min\left(1,{2}\bar{C}
\sqrt{\frac{\log(1/\delta)}{T}} {(\log\log(1/\delta))^2}\right). \]
\end{theorem}

The meta-algorithm also enjoys expected regret guarantees.
\begin{corollary}[Expected simple regret]
\label{cor:expected_sr}
Let the assumptions of Theorem \ref{thm:main} hold. Run
Algorithm~\ref{alg:meta} with \(\delta=e^{-e}\) and with \(\bar C\) chosen large enough as in Theorem~\ref{thm:main}. Then there exists a numerical constant $C_{\rm exp}>0$ such that, for every
$T\ge 1$, Algorithm~\ref{alg:meta} run with the budget not exceeding \(T\) satisfies
\[
    \sup_{f\in\mathcal F}
    \mathbb E\!\left[f(\widehat x)-f^*\right]
    \le \frac{C_{\rm exp}}{\sqrt T}.
\]
\end{corollary}
The proof of Corollary \ref{cor:expected_sr} is deferred to Appendix \ref{appen:expected_sr}.


\section{Proof of Theorem~\ref{thm:main}}
To prove Theorem~\ref{thm:main} we derive separately the bound on the budget in Section \ref{sec:bound_budget} and the bound on the simple regret in Section \ref{sec:bound_error}.

\subsection{Bound on the budget}
\label{sec:bound_budget}
It suffices to consider the case $T> {\bar C}^2 \log(1/\delta)(\log\log(1/\delta))^4$ since otherwise the algorithm stops at the initialization and the bound of the theorem is trivial. This condition will be assumed throughout the proof. Note also that, since \(\Delta_r \geq \bar \Delta\) for $r \in\{0,\dots, R\}$, we have 
\begin{align}\label{eq:glog}
    T  \Delta_{r}^2 \geq 2,
\end{align}
provided that $\bar C\ge \sqrt{2} (\log(3))^{-1/2} (\log\log(3))^{-2}$ (we will assume this condition on $\bar C$ in the sequel).

As \Cref{alg:meta} runs, the range $\Delta_r$ decreases until epoch $R$, the last epoch for which
$\Delta_R\ge\bar\Delta$. After executing epoch \(R\), the algorithm obtains
\(\Delta_{R+1}<\bar\Delta\).
We show that round $R$ is reached using at most $T$ queries if $\bar C>0$ is chosen large enough.

At each round $r\in\{0,1,\dots, R\}$ of \Cref{alg:meta}, the budget spent by calling the Splitting Algorithm (\Cref{alg:splitting}) with the parameters $\Delta_r$ and $\delta_r=\bar\delta(\Delta_r)$ is equal to
\begin{equation}\label{eq:proof:0}
  \bar N(\Delta_r, \bar \delta(\Delta_r)) (2\bar g(\Delta_r)+1)
\le {4}c_{\mathrm{sampl}}^{-2}\frac{\log\Big(2(2 g_r+1)/\delta_r\Big) (2 g_r+1)^3}{\Delta_r^2}
, 
\end{equation} 
where $\bar N$  is the number of queries made at each point of the grid, $2\bar g(\Delta_r)+1= 2g_r+1$ is the number of grid points at round~$r$, and $\delta_r$ is the tolerance error 
at round~$r$. 

Note that if $c_{\mathrm{cond}}\le 1$ then, for any integer $r \in\{1,\dots, R\}$,
\begin{align}\label{eq:proof:1}
  \Delta_{r} = \left(1 - \frac{c_{\mathrm{cond}}}{2(2g_{r-1} + 1)}\right) \Delta_{r-1} \le 
\left(1 - \frac{c_{\mathrm{cond}}}{10 (\log_2(T \Delta_{r-1}^2) )}\right) \Delta_{r-1}.   
\end{align} 
 Indeed, using \eqref{eq:glog} we get
 \( g_{r-1} = \lfloor \log_2(T  \Delta_{r-1}^2) \rfloor + 1 \ge 1   \) and $2(2g_{r-1} + 1)\ge 6$,
which together with the definition \eqref{def:Delta-r} proves the equality in \eqref{eq:proof:1}. To prove the inequality in  \eqref{eq:proof:1}, it suffices to note that 
%
$2(2g_{r-1} + 1) = 4 \lfloor \log_2(T  \Delta_{r-1}^2) \rfloor + 6 \leq 10 \log_2(T  \Delta_{r-1}^2)$. 

Using \eqref{eq:proof:1} and the fact that $(1-1/u)^u \leq e^{-1}$, $\forall u \ge 1$, we obtain:
\begin{equation}\label{eq:proof:2}
 \Delta_{r + r'} \leq  e^{-1}\Delta_{r},   
\end{equation}  
for \( r' \ge \bar{R}(\Delta_r) = \frac{10 \log_2(T \Delta_{r}^2) }{c_{\mathrm{cond}}} \) and $r\in\{0,1,\dots, R\}$.

Now, for integers $m\ge 0$ consider the blocks $\{r \ge 0:  e^{-m-1} < \Delta_r \leq e^{-m}\}$. The values $\Delta_r$ remain of the same order of magnitude within each block. 
It follows from \eqref{eq:proof:2} that 
the cardinality of the $m$th block satisfies: 
$| \{r \ge 0:  e^{-m-1} < \Delta_r \leq e^{-m}\}| \leq \ceil{\bar R(e^{-m})} \leq 2\bar R(e^{-m})\leq \frac{20 \log_2(T e^{-2m})}{c_{\mathrm{cond}}}.$
 Introduce the notation $y_r=T\Delta_r^2$. Then we have 
\begin{equation}\label{eq:proof:3}
 | \{r \ge 0:  Te^{-2m-2} < y_r \leq T e^{-2m}\}|\leq 2\bar R(e^{-m}) \leq \frac{20 \log_2(T e^{-2m})}{c_{\mathrm{cond}}} .   
\end{equation}

Next, we evaluate the expression in \eqref{eq:proof:0}. Plugging in $\delta_r=\delta/y_r$ and noticing that $2g_r+1\le 2\log_2(y_r)+3\le 
5\log_2(y_r) \le 5y_r $ we get
\begin{align*}
  &  \frac{\log\Big(2(2 g_r+1)/\delta_r\Big) (2 g_r+1)^3}{\Delta_r^2} \le  T \frac{\log\Big(2(2\log_2(y_r)+3)/\delta_r\Big) (2\log_2(y_r)+3)^3}{y_r}
  \\
  & \qquad \qquad \qquad \qquad = T \frac{\log\Big(2y_r(2\log_2(y_r)+3)/\delta\Big) (2\log_2(y_r)+3)^3}{y_r}
  \\
   & \qquad \qquad \qquad \qquad \le C_1 T\frac{(\log(1/\delta) + \log_2(y_r))(\log_2(y_r))^3}{y_r},
\end{align*}
where $C_1>0$ is a numerical constant. Combining this bound with \eqref{eq:proof:0}  
we obtain
that the budget $T^*$ consumed before the algorithm stops satisfies
\begin{align*}
T^* &= \sum_{r = 0}^R \bar N(\Delta_r, \bar \delta(\Delta_r)) (2\bar g(\Delta_r)+1)\\ 
&\le 2C_1 c_{\mathrm{sampl}}^{-2} T \sum_{r \geq 0: \Delta_r \geq \bar \Delta} \frac{(\log(1/\delta) + \log_2(y_r))(\log_2(y_r))^3}{y_r}\\ 
&\leq 2C_1 c_{\mathrm{sampl}}^{-2} \sum_{m\geq 0: e^{-m} \geq \bar \Delta} \frac{[\log(1/\delta)+ \log_2(Te^{-2m})](\log_2(Te^{-2m}))^3}{e^{-2m-2}} \bar R(e^{-m}).
\end{align*}
Invoking \eqref{eq:proof:3} we find that, for a numerical constant $C_2>0$, 
\begin{align}\label{eq:proof:5} 
T^* 
&\leq  C_2\Bigg(\log(1/\delta) \sum_{\substack{ m \ge 0:\\ e^{-m} \geq \bar \Delta}}  \frac{(\log_2(Te^{-2m}))^4}{e^{-2m}}
+ \sum_{\substack{ m \ge 0:\\ e^{-m} \geq \bar \Delta}}  \frac{(\log_2(Te^{-2m}))^5}{e^{-2m}}\Bigg).
\end{align}
Using the definition of $\bar \Delta$, the fact that $T \ge \bar{C}^2 \log(3)(\log\log(3))^4$ and carrying out accurate evaluation of the sums in \eqref{eq:proof:5} (see Appendix \ref{appen:budget}) we obtain that $T^*\le T$ if $\bar{C}$ is large enough. 


\subsection{Bound on the regret}
\label{sec:bound_error}

For any executed epoch $r\in\{0,\dots,R\}$, we denote by $\mathcal{E}_r$ the corresponding event in Lemma~\ref{lem:splitting}, which holds with probability at least $1 - \delta_r$. We denote by $\mathcal{E}_*$ the intersection of all these events up to round $R$, and by $\mathcal{E}^c_*$ its complement. Recalling that $ 2\bar R(e^{-m})$ is an upper bound on the number of rounds $r$ such that $ e^{-m-1} < \Delta_r \leq e^{-m}$ 
we have
\begin{align*}
     \mathbb P(\mathcal{E}^c_*)  \leq 2\sum_{r=0}^R \delta_r &\leq  2\sum_{\substack{ m \ge 0:\\ e^{-m} \geq \bar \Delta}}  \bar R(e^{-m}) \bar \delta(e^{-m-1})
    \end{align*}
    since  $\bar \delta$ is a decreasing function. Using~\eqref{eq:proof:3}, the definition of $\bar \delta$, and arguing as in \eqref{eq:proof:6a} (see Appendix \ref{appen:budget}) we find:
\begin{align}\label{eq:proof:11}
 \mathbb P(\mathcal{E}^c_*) 
    \leq  \frac{20 e^2 \delta}{c_{\mathrm{cond}}}\sum_{\substack{ m \ge 0:\\ e^{-m} \geq \bar \Delta}}  
    \frac{\log_2(T e^{-2m})}{Te^{-2m}} 
    \le C_6\delta\left[\frac{\log_2(T)}{T} + \int_{e^{-2}\bar{C}^2 c_\delta}^{+\infty} \frac{\log_2 w}{w^2} dw
    \right],
\end{align}
where $C_6>0$ is a numerical constant.
Here, $c_\delta\ge \log(3)(\log\log(3))^4$. Recalling that  
$T \ge \bar{C}^2 \log(3)(\log\log(3))^4$ we obtain that the expression in the square brackets in~\eqref{eq:proof:11} is smaller than $1/C_6$ if 
$\bar C$ is chosen large enough. Thus, the bound $\mathbb P(\mathcal{E}^c_*)\le \delta$ holds true for $\bar C$ large enough.

\vspace{1mm}

On the event $\mathcal{E}_*$, we can apply  Lemma~\ref{lem:splitting} on each round $r$ of \Cref{alg:meta} until $r=R$, which yields that $\Delta_R \leq \bar \Delta$, and
\begin{align*}
    f(\hat x) - f^* &\leq \Delta_R + \epsilon_{R-1} \leq \bar \Delta + 2\sum_{m\geq 0: e^{-m} \geq \bar \Delta} \bar R(e^{-m}) \psi(e^{-m-1}) e^{-m}
    \end{align*}
    where we used the fact that $\bar R$ is an increasing function, cf.~\eqref{eq:proof:3}, and $\psi$ is a decreasing function, $\psi(\Delta) = 2^{-\bar g(\Delta)+ 2}=2^{- \lfloor \log_2(T\Delta^2)\rfloor+1  }$. Since 
    $\lfloor y\rfloor +1 \ge y$, $\forall y\ge 0$, we get that $\psi(e^{-m-1})\le 2^{- \log_2(Te^{-2m-2})+2 } = 4e^2/(Te^{-2m})$.     
    Using these remarks and~\eqref{eq:proof:3} we obtain that the following inequalities hold for $\bar C$ large enough on the event $\mathcal{E}_*$:  
    \begin{align*}
     f(\hat x) - f^* 
    &\leq \bar \Delta + C_7\sum_{m\ge0: e^{-m} \geq \bar \Delta}    \frac{\log_2(Te^{-2m})}{Te^{-m}} 
    \\
    &\le \bar \Delta + C_8\left[\frac{\log_2(T)}{T} + \frac{1}{\sqrt{T}}\int_{e^{-2}\bar{C}^2 c_\delta}^{+\infty} \frac{\log_2 w}{w^{3/2}} dw 
    \right] 
    \le 2\bar \Delta,
\end{align*}
where the penultimate inequality follows from the same argument as in \eqref{eq:proof:6a} (see Appendix \ref{appen:budget}), $C_7, C_8$~are positive numerical constants, and the final inequality holds under the choice of $\bar C$ large enough due to the fact that $T>  \bar{C}^2 \log(3)(\log\log(3))^4$ . $\square$


\section{Proof of Lemma~\ref{lem:splitting}}


We start by proving two auxiliary lemmas.

\begin{lemma}[Gap between \(f(1/2)\) and the minimum over interval]\label{lem:bound to 1/2}
Let $0\le I_- < 1/2 < I_+\le 1 $.
   Let $f:[I_-,I_+]\to [0,1]$ be a convex function, and consider $\tilde x\in {\rm argmin}_{x \in [I_-,I_+]} f(x).$ Set
   \[
\tilde \Delta := \max_{x \in [I_-,I_+]} f(x) - \min_{x \in [I_-,I_+]} f(x).
\]
Then
\begin{equation}\label{eq:gapdem}
    f(1/2) - f(\tilde x) \leq \frac{\tilde \Delta}{2}.
\end{equation}
\end{lemma}
{\bf Proof.}
Without loss of generality, we assume that \([I_-,I_+] = [0,1]\) and that \(\tilde{x} < 1/2\).
The convexity of $f$ and the fact that $\tilde x \in [0,1/2)$ imply:
    \[
    \frac{f(1/2) - f(\tilde{x})}{1/2 - \tilde x} \leq \frac{f(1) - f(1/2)}{1/2}.
    \]
   Since $1/2 - \tilde x \leq 1/2$, we obtain $
    f(1/2) - f(\tilde{x}) \leq f(1) - f(1/2),
    $
    so that 
    $
    2(f(1/2) - f(\tilde{x}))\leq f(1) - f(\tilde{x}) \le \tilde\Delta.
    $
Thus, \eqref{eq:gapdem} follows. $\square$

\begin{lemma}\label{lem:concentration}
Let $c_{\mathrm{sampl}}>0$, $\delta\in (0,1)$. Consider the event 
$\mathcal{E} = \Big\{|\hat f_x - f(x)| \leq   \frac{c_{\mathrm{sampl}}\Delta}{2g+1}, \forall x\in \mathcal G\Big\}.$ 
    We have $\mathbb{P}(\mathcal{E})\ge 1-\delta.$
\end{lemma}
{\bf Proof.} By the definition of the algorithm, for each $x\in \mathcal G$ we have $\hat f_x - f(x)=\frac{1}{N}\sum_{j=1}^N \varepsilon_{ix}$, where $\varepsilon_{ix}$ are independent $1$-subGaussian random variables.
Thus, each $\hat f_x - f(x)$ is a $1/\sqrt{N}$-subGaussian random variable, so that the union bound implies: 
\[
\mathbb P(\mathcal E^c)
=
\mathbb P\left(
\exists x\in\mathcal G:
|\hat f_x-f(x)|>
c_{\mathrm{sampl}}\frac{\Delta}{2g+1}
\right)
\le
2|\mathcal G|
\exp\left(
-\frac N2
\left(
\frac{c_{\mathrm{sampl}}\Delta}{2g+1}
\right)^2
\right).
\]
By the definition of $N$ in \eqref{def:N} and the fact that $|\mathcal G| = 2g+1$ we obtain that the right hand side of the above display does not exceed $\delta$. $\square$


\paragraph{Proof of Lemma~\ref{lem:splitting}.} 
Without loss of generality assume that \([I_-,I_+] = [0,1]\).
We place ourselves on the event $\mathcal{E}$ and consider separately the three possible cases. 


{\bf Case 1: $\mathcal G_-\ne \emptyset$ and $\max \mathcal G_->0$.}
In this case, we have $I_-' = \max \mathcal G_-$ and $I_-' \in (0,1/2)$. By the definition of $\mathcal G$, it follows that $2I_-' \in \mathcal G$ and $2I_-' \in (0,1/2]$. Also, due to the definition of $\mathcal G_-$, 
\[\hat f_{I_-'} - \hat f_{2I_-'} > c_{\mathrm{cond}}\frac{\Delta}{2g+1}.\]
On the event $\mathcal{E}$, this implies the bound
\[f(I_-') - f(2I_-') > c_{\mathrm{cond}}\frac{\Delta}{2g+1} - 2c_{\mathrm{sampl}}\frac{\Delta}{2g+1} \ge c_{\mathrm{cond}}\frac{\Delta}{2(2g+1)},\]
where we used the condition $c_{\mathrm{sampl}} \leq c_{\mathrm{cond}}/4$. We deduce that 
$\tilde x\ge  I_-'$ as $f$ is decreasing on the right of~$I_-'$.
Using the convexity of $f$ we obtain: 
\[f(0) - f(I_-')\ge f(I_-') - f(2I_-') > c_{\mathrm{cond}}\frac{ \Delta}{2(2g+1)}.\]
It follows that 
\begin{align*}
    f(I_-') - f(\tilde x)  =  f(I_-') - f(0) + f(0) - f (\tilde x) 
     \leq - \frac{c_{\mathrm{cond}} \Delta}{2(2g+1)} +  \tilde \Delta 
    \leq \left(1 - \frac{c_{\mathrm{cond}}}{2(2g+1)}\right) \Delta,
\end{align*}
where we used the bound $\tilde \Delta \leq \Delta$ granted by Assumption$(\epsilon,\Delta,[I_-, I_+])$. Invoking again the convexity of $f$ and using~\eqref{eq:gapdem} we find that, for any $y \in [I_-', 1/2]$,
\begin{align*}
f(y) - f(\tilde x) &\leq \frac12(f(I_-') - f(\tilde x)) + \frac12(f(1/2)-f(\tilde x))
\\
& \le \frac12\left[\left(1 - \frac{c_{\mathrm{cond}}}{2(2g+1)}\right) \Delta
+
\frac{\tilde \Delta}{2}\right]
\leq \left(1 - \frac{c_{\mathrm{cond}}}{2(2g+1)}\right) \Delta.
\end{align*}
\noindent
\underline{In summary,} the following facts hold on the event $\mathcal{E}$ in Case 1.\\
(i) For any $y \in [I_-', 1/2]$ we have:
\[f(y) - f(\tilde x) \leq \Big(1 - \frac{c_{\mathrm{cond}}}{2(2g+1)}\Big) \Delta.\]
 { (ii) If $\tilde x\le 1/2$ then
 $\min_{x\in [I_-',1/2]}f(x) - f(\tilde x) =0.$
 Indeed, recall that we also have \\ $I_-'\le \tilde x$, so that $\tilde x\in [I_-',1/2]$. Moreover, $f(\tilde x)=\underset{x\in [0,1]}{\min}f(x)$.}
 
\vspace{2mm}

{\bf Case 2: $\mathcal G_-\ne \emptyset$ and $\max \mathcal G_-=0$ .}

In this case 
$I_-' =2^{-g}$. Thus, by the definition 
of $\mathcal G_-$, 
\[\hat f_{0} - \hat f_{2^{-g}} > c_{\mathrm{cond}}\frac{\Delta}{2g+1},\]
while for any $x\in \mathcal G$ such that $I_-' \leq x<1/2$:
\[\hat f_{x} - \hat f_{2x} \leq c_{\mathrm{cond}}\frac{\Delta}{2g+1}.\]
Therefore, on the event $\mathcal{E}$ 
for any $x\in \mathcal G$ such that $I_-' \leq x<1/2$ we have:
\[f(x) - f(2x) \leq c_{\mathrm{cond}}\frac{\Delta}{2g+1} + 2c_{\mathrm{sampl}}\frac{\Delta}{2g+1} \leq  3c_{\mathrm{cond}}\frac{\Delta}{2(2g+1)},\]
since $c_{\mathrm{sampl}} \leq c_{\mathrm{cond}}/4$.
Hence, by the convexity of $f$ we obtain that,
for any $y \in [I_-',1/2]=[2^{-g},1/2]$, 
\begin{align}
 f(y) - f(1/2) &\leq  f(2^{-g}) - f(1/2) = \sum_{i=2}^g(f(2^{-i}) -f(2^{-i+1}))
 \nonumber \\
 &\le 3c_{\mathrm{cond}}\frac{\Delta (g-1)}{2(2g+1)} \le 3c_{\mathrm{cond}}\frac{\Delta}{4}.\label{eq:lem}
\end{align}
This inequality together with~\eqref{eq:gapdem} implies that, for any $y \in [I_-',1/2]$,
\begin{align}\label{eq:case2}
f(y) - f(\tilde x) \leq    3c_{\mathrm{cond}}\frac{\Delta}{4} + \frac{\tilde \Delta}{2} \leq \frac{3\Delta}{4},
\end{align}
since $c_{\mathrm{cond}} \leq 1/3$ and $\tilde \Delta \leq \Delta$ by Assumption$(\epsilon,\Delta,[I_-, I_+])$.

Our next aim is to prove a bound on $\min_{x\in [I_-',1/2]}f(x) - f(\tilde x)$ assuming that $\tilde x\le 1/2$. If $\tilde x\in [I_-',1/2]$ we obviously have $\min_{x\in [I_-',1/2]}f(x) - f(\tilde x)=0$. 
Assume that $\tilde x < I_-'=2^{-g}$. Then by convexity of $f$ and since the range of $f$ is bounded by $\Delta$ in $[I_-,I_+] = [0,1]$ the following holds:
\[f(I_-') - f(\tilde x) =f(2^{-g}) - f(\tilde x) \leq 2^{-g+2} [f(1/2) - f(1/4)] \leq 2^{-g+2} \Delta.\]

{ Therefore, for $\tilde x < I_-'$ we have
$
 \min_{x\in [I_-',1/2]}f(x) - f(\tilde x) \le 2^{-g+2} \Delta.   
$
}

\noindent
\underline{In summary,} the following facts hold on the event $\mathcal{E}$ in Case 2.\\
(i) For any $y \in [I_-',1/2]$
\[f(y) - f(\tilde x) \leq \frac{3\Delta}{4}.
\]
{ (ii) If $\tilde x\le 1/2$ then we have 
  $\min_{x\in [I_-',1/2]}f(x) - f(\tilde x) \le 2^{-g+2} \Delta.$
}



\paragraph{Case 3:  $\mathcal G_- = \emptyset$.} In this case, we have $I_-' =0$ and, by the definition of $\mathcal G_-$,
$$\hat f_{0} - \hat f_{2^{-g}} \leq  c_{\mathrm{cond}}\frac{\Delta}{2g+1},
\quad \text{and} \quad \hat f_{x} - \hat f_{2x} \leq c_{\mathrm{cond}}\frac{\Delta}{2g+1}
$$
for all $x\in \mathcal G$ such that $I_-' \leq x<1/2$.
Hence, on the event $\mathcal{E}$ we have, similar to {Case 2}, cf.~\eqref{eq:lem}, that  for any $y \in [I_-',1/2] = [0,1/2]$,
\begin{align*}
 f(y) - f(1/2) &\leq  f(0) - f(1/2) = f(0) - f(2^{-g}) + \sum_{i=2}^g(f(2^{-i}) -f(2^{-i+1}))
 \nonumber \\
 &\le 3c_{\mathrm{cond}}\frac{ g\Delta }{2(2g+1)} \le 3c_{\mathrm{cond}}\frac{\Delta}{4}.
\end{align*}
By the same argument as in  Case 2, this implies \eqref{eq:case2} for any $y \in [I_-',1/2]=[0,1/2]$.

\vspace{1mm}

\noindent
\underline{In summary}, the following facts hold on the event $\mathcal{E}$ in Case 3.\\
(i) For any $y \in [I_-',1/2]$
\[f(y) - f(\tilde x) \leq   \frac{3\Delta}{4}.\]
(ii) {If $\tilde x\le 1/2$ we have, since $I_-' = 0$,
$\min_{x\in [I_-',1/2]}f(x) - f(\tilde x)=0.$
}

\vspace{2mm}

Putting together the conclusions obtained in the three cases
we deduce that  
under $\mathcal{E}$ the following two facts hold.

\noindent
{\bf Fact (a):}
\[\max_{y \in [I_-',1/2]}f(y) - f(\tilde x) \leq    \Big[\frac{3}{4} \lor \Big(1 - \frac{c_{\mathrm{cond}}}{2(2g+1)}\Big)\Big] \Delta = \Delta'.\]
{\bf Fact (b):} { \it If $\tilde x\le 1/2$ we have:}
\[\min_{x\in [I_-',I_+']}f(x) - f(\tilde x) \le \min_{x\in [I_-',1/2]}f(x) - f(\tilde x) \leq 2^{-g+2} \Delta.\]
From Fact (b) and Assumption$(\epsilon,\Delta,[I_-, I_+])$ we obtain : 
\begin{equation}\label{factb}
   \tilde x\le 1/2 \quad \Rightarrow \quad \min_{x\in [I_-',I_+']}f(x) - f^*  \leq \epsilon +  2^{-g+2} \Delta = \epsilon'. 
\end{equation}
An analogous argument dealing with the interval $[1/2,1]$ and $I_+'$ instead of the interval $[0,1/2]$ and $I_-'$ yields that
under $\mathcal{E}$ the next two facts hold.

\noindent
{\bf Fact (c):}
\[\max_{y\in [1/2,I_+']}f(y) - f(\tilde x) \leq    \Big[\frac{3}{4} \lor \Big(1 - \frac{c_{\mathrm{cond}}}{2(2g+1)}\Big)\Big] \Delta = \Delta'.\]
{\bf Fact (d):} 
\[\tilde x\ge 1/2 \quad \Rightarrow \quad \min_{x\in [I_-',I_+']}f(x) - f^*  \leq \epsilon +  2^{-g+2} \Delta = \epsilon'.\] 
Combining Facts (a) and (c) we get the first inequality of Lemma \ref{lem:splitting}. Combining~\eqref{factb} and Fact (d) we get the second inequality of Lemma \ref{lem:splitting}. $\square$

\paragraph{Acknowledgements.} 

The work of A. Carpentier was partially supported by the Deutsche Forschungsgemeinschaft (DFG) through SFB 1294 “Data Assimilation,” Project A03 (Project ID 318763901); the DFG Research Unit FOR 5381 “Mathematical Statistics in the Information Age—Statistical Efficiency and Computational Tractability,” Project TP 02 (Project ID 460867398); the Université franco-allemande (UFA) through the Collège doctoral franco-allemand CDFA-02-25 “Statistisches Lernen für komplexe stochastische Prozesse”; and the European Research Council (ERC) through the ERC Consolidator Grant SOCE (Grant No. 101229569). Views and opinions expressed are, however, those of the authors only and do not necessarily reflect those of the European Union or the European Research Council. Neither the European Union nor the granting authority can be held responsible for them.

The work of Chloé Rouyer was also partially supported by the Deutsche Forschungsgemeinschaft (DFG) through SFB 1294 “Data Assimilation,” Project A03 (Project ID 318763901).

The work of Alexandre B. Tsybakov was supported by Labex ECODEC (ANR-11-LABEX-0047) and ANR MaLIP (ANR-25-CE40-3228-01).

The work of Arya Akhavan was funded by UK Research and Innovation (UKRI) under the UK Government’s Horizon Europe funding guarantee (grant number EP/Y028333/1).
{
\bibliographystyle{alpha}

\bibliography{biblio}
}


\appendix

\section{Appendix}
\label{sec:appendix}

\subsection{Complement to the Bound on the Budget}
\label{appen:budget}
{\bf Evaluation of the sums in formula \eqref{eq:proof:5}.}
Note that if $\bar C$ is chosen large enough the map $m\mapsto \frac{(\log_2(Te^{-2m}))^5}{e^{-2m}}$ is increasing on the set of $m\ge 0$ such that $e^{-m} \geq \bar \Delta$. We have 
\begin{align}
\sum_{\substack{ m \ge 0:\\ e^{-m} \geq \bar \Delta}}  \frac{(\log_2(Te^{-2m}))^5}{e^{-2m}}
& = (\log_2(T))^5 + \sum_{\substack{ m \ge 1:\\ e^{-m} \geq \bar \Delta}}  \frac{(\log_2(Te^{-2m}))^5}{e^{-2m}}
\nonumber
\\
& \leq (\log_2(T))^5 + \int_{1}^{\log(1/\bar \Delta)+1}  \frac{(\log_2(Te^{-2u}))^5}{e^{-2u}} du
\nonumber
\\
& = (\log_2(T))^5 +    (T/2)\int_{e^{-2}\bar{C}^2 c_\delta}^{e^{-2}T} (\log_2 w)^5/w^2 dw \quad \text{(change of variable $w=Te^{-2u}$)}\nonumber
\\
& \leq (\log_2(T))^5 +    (T/2)\int_{e^{-2}\bar{C}^2 c_\delta}^{+\infty} (\log_2 w)^5/w^2 dw
\label{eq:proof:6a}
\end{align}
where $c_\delta=\log(1/\delta)(\log\log(1/\delta))^4\ge \log(3)(\log\log(3))^4$. We may also recall that 
$T> {\bar C}^2 c_\delta \ge \bar{C}^2 \log(3)(\log\log(3))^4$. Thus, for $\bar{C}$ large enough the inequality $C_2 (\log_2(T))^5\le T/4$ holds true. Furthermore, choosing $\bar C>0$ large enough makes the last integral in \eqref{eq:proof:6a} smaller than $(2C_2)^{-1}$. 
Combining these arguments, we get that 

\begin{align} \label{eq:proof:6}
C_2 \sum_{\substack{ m \ge 0:\\ e^{-m} \geq \bar \Delta}}  \frac{(\log_2(Te^{-2m}))^5}{e^{-2m}} \le \frac{T}{2}
\end{align}
whenever $\bar C>0$ is large enough. Quite analogously, for the first sum in \eqref{eq:proof:5} we have
\begin{align*}
\log(1/\delta)\hspace{-1mm}\sum_{\substack{ m \ge 0:\\ e^{-m} \geq \bar \Delta}}  \frac{(\log_2(Te^{-2m}))^4}{e^{-2m}}
\leq \log(1/\delta) \Big((\log_2(T))^4 +    \frac{T}{2}\int_{e^{-2}\bar{C}^2 c_\delta}^{+\infty}
\frac{(\log_2 w)^4}{w^2} dw\Big).
\end{align*}
There exist absolute constants $C_3, C_4>0$ such that 
\begin{align}
\log(1/\delta)\int_{e^{-2}\bar{C}^2 c_\delta}^{+\infty}
\frac{(\log_2 w)^4}{w^2} dw
&\le C_3 \log(1/\delta)\frac{(\log (\bar{C}^2 c_\delta))^4}{\bar{C}^2 c_\delta} 
\nonumber
\\
\label{eq:proof:9a}
&\le C_4 \frac{\log(1/\delta)(\log\log(1/\delta))^4}{\bar{C}^2 c_\delta} = \frac{C_4}{\bar{C}^2}.
\end{align}

Furthermore, since $T> {\bar C}^2 \log(1/\delta)(\log\log(1/\delta))^4$ we have that, for $\bar C>0$ large enough the map $T\mapsto \frac{(\log_2(T))^4}{T}$ is decreasing and
\begin{align}\label{eq:proof:9b}
\log(1/\delta) \frac{(\log_2(T))^4}{T} \le   C_5 \frac{(\log(\bar C)+ \log\log(1/\delta)+ \log\log\log(1/\delta))^4}{{\bar C}^2 (\log\log(1/\delta))^4},
\end{align}
where $C_5>0$ is an absolute constant.
By choosing $\bar C>0$ large enough the right hand sides of \eqref{eq:proof:9a} and \eqref{eq:proof:9b} can be rendered  smaller than $(4C_2)^{-1}$. 
Therefore, we conclude that
\begin{align} \label{eq:proof:7}
C_2 \log(1/\delta) \sum_{\substack{ m \ge 0:\\ e^{-m} \geq \bar \Delta}}  \frac{(\log_2(Te^{-2m}))^4}{e^{-2m}} \le \frac{T}{2}
\end{align}
if $\bar C>0$ is chosen large enough. Combining \eqref{eq:proof:5}, \eqref{eq:proof:6}, and \eqref{eq:proof:7} yields that $T^*\le T$ if $\bar C>0$ is chosen large enough.

\subsection{Analysis of the Expected Simple Regret}
\label{appen:expected_sr}

We now provide a proof for Corollary \ref{cor:expected_sr}.

\paragraph{Proof of Corollary \ref{cor:expected_sr}}
For any $f \in \mathcal F$, we can fix $\delta_*=e^{-e}<1/3$, and run Algorithm~\ref{alg:meta} with
$\delta=\delta_*$. Then
\[
    \bar\Delta
    =
    \bar C
    \sqrt{\frac{\log(1/\delta_*)}{T}}
    \bigl(\log\log(1/\delta_*)\bigr)^2
    =
    \frac{C_0}{\sqrt T},
\]
where $C_0>0$ is a numerical constant. Increasing $\bar C$, if necessary,
we may assume that $C_0$ is large enough for the budget bound in
Theorem~\ref{thm:main} to hold. Thus the algorithm uses at most $T$ oracle
calls. If $\bar\Delta>1$, the claim follows after increasing $C_{\rm exp}$,
since $0\le f(\widehat x)-f^*\le 1$. Hence assume $\bar\Delta\le 1$.

Let \(R\) be the last executed epoch, so that the calls of the Splitting
Algorithm are indexed by \(r=0,\ldots,R\), and the returned interval has index
\(R+1\). For every such $r$, since
$\Delta_r\ge \bar\Delta=C_0/\sqrt T$, choosing $C_0$ large enough gives
\[
    T\Delta_r^2\ge C_0^2\ge 2,
    \qquad
    \delta_r=\frac{\delta_*}{T\Delta_r^2}
    \le \frac{\delta_*}{C_0^2}<\frac12.
\]
Thus $g_r\ge 2$, and Lemma~\ref{lem:splitting} holds at each executed epoch. Write $I^{(r)} := [I_-^{(r)}, I_+^{(r)}]$ for the interval maintained by Algorithm~\ref{alg:meta} at epoch $r$. Let \(\mathcal E_r\) be the event that the call to Algorithm~\ref{alg:splitting}
at epoch \(r\), with input \((I^{(r)},\Delta_r,\delta_r)\), outputs an interval
\(I^{(r+1)}\) satisfying
Assumption\((\epsilon_{r+1},\Delta_{r+1},I^{(r+1)})\). For \(r=0,\ldots,R\), define
\[
    A_r=\bigcap_{s=0}^{r-1}\mathcal E_s,
    \qquad
    F_r=A_r\cap \mathcal E_r^c ,
\]
with \(A_0\) the sure event. Also set $ A_{R+1}=\bigcap_{s=0}^{R}\mathcal E_s$. On \(A_r\), the induction from the proof of Theorem~\ref{thm:main} gives
\[
    \min_{x\in I^{(r)}} f(x)-f^*\le \epsilon_r,
    \qquad
    \max_{x,y\in I^{(r)}}(f(x)-f(y))\le \Delta_r .
\]
Hence every \(x\in I^{(r)} \) satisfies
\[
    f(x)-f^*\le \epsilon_r+\Delta_r .
\]
Let \(\mathscr F_r\) denote the sigma-field generated by all oracle queries
and observations before epoch \(r\). On \(A_r\), the interval \(I^{(r)}\) is
\(\mathscr F_r\)-measurable and satisfies
Assumption\((\epsilon_r,\Delta_r,I^{(r)})\). Moreover, the oracle noises used
during epoch \(r\) are independent of \(\mathscr F_r\). Therefore, by
Lemma~\ref{lem:splitting}, on $A_r$ we have
\[
    \mathbb P(\mathcal E_r^c\mid \mathscr F_r)\le \delta_r.
\]
Consequently,
\[
    \mathbb P(\mathscr F_r)
    =
    \mathbb E\!\left[
        \mathbf 1_{A_r}
        \mathbb P(\mathcal E_r^c\mid \mathscr F_r)
    \right]
    \le \delta_r .
\]

Since the intervals are nested, on $\mathscr F_r$ the final output still belongs to
$I^{(r)}$. On $A_R$, the same induction gives the bound at scale $R$. Hence

\begin{align*}
    \mathbb E\!\left[f(\widehat x)-f^*\right]
    &\le
    \epsilon_{R+1}+\Delta_{R+1}
    +
    \sum_{r=0}^{R}\mathbb P(F_r)(\epsilon_r+\Delta_r)  
   \\ 
   & \le
    \epsilon_{R+1}+\Delta_{R+1}
    +
\sum_{r=0}^{R}\delta_r(\epsilon_r+\Delta_r). \numberthis \label{proof_exp:bound_exp_sr}
\end{align*}

It remains to bound the deterministic terms. By the stopping rule,
\begin{equation}
    \Delta_{R+1}<\bar\Delta=\frac{C_0}{\sqrt T}. \label{proof_exp:bound_DeltaR+1}
\end{equation}

Also,
\[
    \psi(\Delta)
    =
    2^{-\bar g(\Delta)+2}
    =
    2^{1-\lfloor \log_2(T\Delta^2)\rfloor}
    \le \frac{4}{T\Delta^2}.
\]
Therefore, for every \(r\le R+1\),
\[
    \epsilon_r
    =
    \sum_{s=0}^{r-1}\psi(\Delta_s)\Delta_s
    \le
    4\sum_{s=0}^{r-1}\frac{1}{T\Delta_s}.
\]
In order to bound the quantities in Equation \eqref{proof_exp:bound_exp_sr},  we consider the following two sums:

\[
    S_1=\sum_{r=0}^{R}\frac{1}{T\Delta_r},
    \qquad
    S_2=\sum_{r=0}^{R}\frac{1}{T\Delta_r^2}.
\]
We can bound the dependency on $\Delta_r$ by taking advantage of the block-counting estimate from Equation \eqref{eq:proof:3} in the proof of Theorem~\ref{thm:main}, which states that
\[
    \left|\{r\leq R:e^{-m-1}<\Delta_r\le e^{-m}\}\right|
    \le C_b\bigl(1+\log(Te^{-2m})\bigr)
\]
for every relevant $m\ge0$ and some constant $C_b > 0$. Since $\bar\Delta=C_0/\sqrt T$, summing over these blocks yields
\[
    S_1\le \frac{C}{\sqrt T},
    \qquad
    S_2\le C,
\]
for a numerical constant $C>0$. Thus
\begin{equation}
     \epsilon_{R+1}+\Delta_{R+1}
    \le
    4S_1+\bar\Delta
    \le
    \frac{4C + C_0}{\sqrt T}. \label{proof_exp:bound_DeltaR+1_epsR+1}
\end{equation}
    
Moreover to bound the sum in Equation \eqref{proof_exp:bound_exp_sr}, we have:
\[
    \sum_{r=0}^{R}\delta_r\Delta_r
    =
    \delta_* S_1
    \le
    \frac{C}{\sqrt T},
\]
and, since $\sup_{r\le R}\epsilon_r\le 4S_1$,
\begin{equation}
    \sum_{r=0}^{R}\delta_r\epsilon_r
    \le
    \Bigl(\sup_{r\le R}\epsilon_r\Bigr)
    \sum_{r=0}^{R}\delta_r
    \le
    S_1\delta_* S_2
    \le
    \frac{(4/3)C^2}{\sqrt T},
 \label{proof_exp:bound_sum}
\end{equation}
where the second inequality follows from the construction of $\delta_r$ and the third from $\delta^* \leq 1/3$.
Plugging in Equations \eqref{proof_exp:bound_DeltaR+1_epsR+1} and \eqref{proof_exp:bound_sum} in Equation \eqref{proof_exp:bound_exp_sr} with a large enough constant $C_{\rm exp}$ yields
\[
    \mathbb E\!\left[f(\widehat x)-f^*\right]
    \le
    \frac{C_{\rm exp}}{\sqrt T}.
\]
Taking the supremum over $f\in\mathcal F$ completes the proof. 
$\square$


\newpage

\end{document}